\documentclass[11pt]{article}
\usepackage{amsmath, amsthm, amssymb,latexsym, enumerate,kotex}
\usepackage{graphicx}
\usepackage{hyperref}
\usepackage{xcolor}
\usepackage{cite}
\usepackage{dsfont}
\usepackage{graphicx, caption, subcaption}
\usepackage{subcaption} 
\usepackage{caption}
\usepackage{colortbl}
\usepackage{float}
\usepackage[normalem]{ulem}
 \usepackage{tikz}
\usepackage{graphicx, graphics,psfrag}
\usetikzlibrary{decorations.pathreplacing}
\tikzstyle{c} = [draw, every node/.style={circle, draw, inner sep=1.5pt,fill=white}]
\tikzstyle{d} = [draw, every node/.style={circle, draw, inner sep=2pt,fill=white,font=\small}]

\usepackage{mathdots}
\usetikzlibrary{patterns,calc}
\tikzset{
super thick/.style= {line width=2.4pt}
}

\numberwithin{equation}{section}
\usepackage{array}
\newcolumntype{L}[1]{>{\raggedright\let\newline\\\arraybackslash\hspace{0pt}}m{#1}}
\newcolumntype{C}[1]{>{\centering\let\newline\\\arraybackslash\hspace{0pt}}m{#1}}
\newcolumntype{R}[1]{>{\raggedleft\let\newline\\\arraybackslash\hspace{0pt}}m{#1}}

\newcommand{\modthree}[1]{\mathrm{(}#1 \mathrm{\ mod\  3)}}

\newcommand{\modfour}[1]{\mathrm{(}#1 \mathrm{\ mod\  4)}}

\definecolor{LemonChiffon}{rgb}{100, 98, 80}
\definecolor{myblue}{rgb}{0,0.4,0.8}
\definecolor{mygreen}{rgb}{0, 0.8, 0.4}
\definecolor{myred}{rgb}{204, 0, 0}
\definecolor{violet}{RGB}{0.4,0.2,1}
\definecolor{brown}{rgb}{0.6, 0.4, 0}

\newtheorem{theorem}{Theorem}[section]
\newtheorem{lemma}[theorem]{Lemma}
\newtheorem{corollary}[theorem]{Corollary}

\newtheorem{claim}[theorem]{Claim}

\newtheorem{conjecture}[theorem]{Conjecture}

\newtheorem{observation}[theorem]{Observation}
\newtheorem{definition}[theorem]{Definition}
\newtheorem{fact}[theorem]{Fact}

\title{Existence of cycles of length divisible by $3$ or $4$}
\author{Ilkyoo Choi,\quad Hojin Chu,\quad Ringi Kim,\quad Boram Park}

\author{Ilkyoo Choi
\thanks{Department of Mathematics, Hankuk University of Foreign Studies, Yongin-si, Gyeonggi-do 17035, Republic of Korea.
 E-mail: \texttt {ilkyoo@hufs.ac.kr},
  Discrete Mathematics Group, Institute for Basic Science (IBS), Daejeon 34126, Republic of Korea, and 
  Korea Institute for Advanced Study (KIAS), Seoul 02455, Republic of Korea.}
\and Hojin Chu
\thanks{School of Computational Sciences, Korea Institute for Advanced Study (KIAS), Seoul 02455, Republic of Korea. E-mail: {\tt hojinchu@kias.re.kr}.}
\and Ringi Kim
\thanks{Department of Mathematics, Inha University, Incheon 22212, Republic of Korea. E-mail: {\tt ringikim@inha.ac.kr}.}
\and Boram Park
\thanks{Department of Mathematics Education, Seoul National University, Seoul 08826, Republic of Korea. E-mail: {\tt borampark@snu.ac.kr}.}
}

\date{\today}

\begin{document}

\maketitle 

\begin{abstract} 
Dean conjectured that for each integer $k \ge 3$, every graph with minimum degree at least $k$ has a cycle whose length is divisible by $k$; this conjecture is known to be true for all $k\neq 5$.
For $k\in\{3,4\}$, stronger statements are true: every graph with minimum degree at least $2$ and at most $k-2$ vertices of degree $2$ has a cycle whose length is divisible by $k$. 

We further strengthen these results by characterizing all graphs with minimum degree at least $2$ and at most three vertices of degree $2$ that have no cycle of length divisible by $k$, for each $k\in\{3,4\}$.
As a corollary, we obtain that every graph with minimum degree at least $2$ and at most two vertices of degree $2$ has a cycle whose length is divisible by 3, and that every graph on at least nine vertices with minimum degree at least $2$ and at most three vertices of degree $2$ has a cycle whose length is divisible by 4.    
\end{abstract}

\section{Introduction}
All graphs in this paper are finite and simple. 
Given a graph $G$ and a vertex $v$, the degree and neighborhood of $v$ are denoted by $d_G(v)$ and $N_G(v)$, respectively.
A {\it $d$-vertex} is a vertex of degree $d$. 
The set of $2$-vertices of a graph $G$ is denoted by $V_2(G)$.
For two integers $k$ and $\ell$, an {\it $(\ell \bmod k)$-cycle} and an {\it $(\ell \bmod k)$-path} are a cycle and a path, respectively, of length $m$ such that $m\equiv \ell\pmod{k}$; 
they are also referred to as a cycle and a path, respectively, of length $(\ell \bmod k)$.

The study of cycle lengths is a central theme in graph theory, with a long history of results related to the existence of cycles of prescribed lengths, see \cite{Bondy1996,V2016}.
A fundamental line of research seeks to understand how structural conditions, such as large average degree~\cite{EG59, ES82, BS74, AKSV14}, large minimum degree~\cite{Diwan10, Thomassen1988'}, large chromatic number~
\cite{Gyarfas92,Tuza13,LRS10,KSV17}, or strong connectivity~\cite{CS94, DLS93, GHLM19}, force the existence of cycles of prescribed lengths. A particularly well-studied direction concerns cycles whose lengths lie in prescribed congruence classes modulo an integer $k$. 
In particular, cycles whose lengths are divisible by $k$ have attracted considerable attention. 

Early conjectures of Burr and Erd\H{o}s~\cite{BurrErdos} and  work of Thomassen~\cite{Thomassen1983} suggest that sufficiently strong global conditions should guarantee the existence of cycles in all congruence classes modulo $k$, highlighting the fundamental role of extremal and minimum degree conditions in forcing such cycles, see, for example, \cite{BBollobas, LM18, GHLM19}.
While these problems aim to guarantee cycles in all congruence classes, it is natural to study existence of cycles in a single prescribed class. 
A classical direction in extremal graph theory is to relax underlying assumptions while still ensuring the existence of such cycles. 
For instance, see~\cite{Thomassen1988, SV17, CPR2025}.
A natural starting point in this direction is the following conjecture of Dean~\cite{Dean1988} from 1988.

\begin{conjecture}[Dean's conjecture]
    For every integer $k \ge 3$, every graph with minimum degree at least $k$ contains a cycle of length divisible by $k$.
\end{conjecture}

The conjecture was verified for $k\in\{3,4\}$ in \cite{CS94, DLS93} shortly after it was proposed, and more recently for all $k \ge 6$~\cite{LMZ2026}. 
In this paper, we dive deeper into the cases $k=3$ and $k=4$. Notably, the known results for these values already hold under assumptions weaker than those in Dean's conjecture, allowing a small number of $2$-vertices.
Below are the results by Chen and Saito~\cite{CS94} and Dean, Lesniak, and Saito~\cite{DLS93} establishing the cases $k=3$ and $k=4$, respectively. 

% More precisely, Chen and Saito~\cite{CS94} established the case $k=3$ by proving the following theorem.
\begin{theorem}[\!\!\cite{CS94}]\label{thm:previous:0mod3}
Every graph $G$ with $\delta(G)\ge 2$ and $|V_2(G)|\le 1$ contains a $\modthree{0}$-cycle. 
\end{theorem}

% Dean, Lesniak, and Saito~\cite{DLS93} established the case $k=4$ by proving the following  result.

\begin{theorem}[\!\cite{DLS93}]\label{DLS93thm}
Every graph $G$ with $\delta(G)\ge 2$ and $|V_2(G)|\le 2$ contains a $\modfour{0}$-cycle.     
\end{theorem}

These results naturally lead to the question of whether cycles of length divisible by $3$ or $4$ must still exist when more $2$-vertices are allowed.
In this paper, we answer this question by  completely characterizing  graphs with minimum degree at least $2$ and at most three $2$-vertices that have no cycle of length divisible by $3$ (Theorem~\ref{thm:main:0mod3}), and, separately, those that have no cycle of length divisible by $4$ (Theorem~\ref{thm:main:0mod4}). 
As a corollary,  every graph with minimum degree at least $2$ and at most two $2$-vertices has both a $\modthree{0}$-cycle and a $\modfour{0}$-cycle.

Subsections 1.1 and 1.2 present our main results on $\modthree{0}$-cycles and $\modfour{0}$-cycles, respectively. 

\subsection{Graphs with a cycle of length divisible by $3$}

Our first main result characterizes all graphs with minimum degree at least $2$ and at most three $2$-vertices that have no $\modthree{0}$-cycle.
% shows that we can allow one more $2$-vertex than in the bound of Theorem~\ref{thm:previous:0mod3}, $|V_2(G)|\le 1$ is not tight, that is,  a graph with more $2$-vertices still contains a $\modthree{0}$-cycle unless it belongs to a certain class of graphs. 
There are infinitely many such graphs, which we call {\it exceptional graphs}.
We show that every exceptional graph can be generated from $K_{2, 3}$ by applying three operations, defined in  Definition~\ref{def:exceptional}.
% We call such a graph an exceptional graph, and we completely characterize the family, where the
We now state our first main result.

\begin{theorem}\label{thm:main:0mod3}
Let $G$ be a graph with $\delta(G)\ge 2$ and $|V_2(G)|\le 3$.
Then $G$ has a $\modthree{0}$-cycle if and only if $G$ is not an exceptional graph.
\end{theorem}

Every exceptional graph has three $2$-vertices, so the following corollary is immediate; note that this corollary enlarges the class of graphs guaranteed to have a $\modthree{0}$-cycle  by Theorem~\ref{thm:previous:0mod3}.

\begin{corollary}
Let $G$ be a graph with $\delta(G)\ge 2$.
If $|V_2(G)|\le 2$, then $G$ has a $\modthree{0}$-cycle.    
\end{corollary}

We now define the class of exceptional graphs, a recursively constructed obstruction family obtained from $K_{2,3}$ while preserving exactly three $2$-vertices; see Figure~\ref{fig:exceptional} for examples.

\begin{figure}[h!]
\centering
\includegraphics[page=2,height=3.5cm]{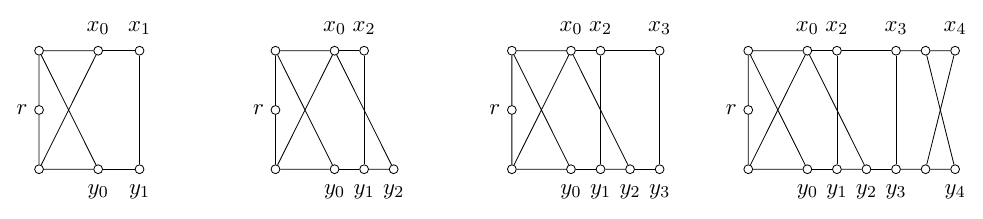}
\caption{Some exceptional graphs with root $r$ }\label{fig:exceptional}
\end{figure}

\begin{definition}\label{def:exceptional}\rm
Let $r$ be a $2$-vertex of a graph $G$.
We say $G$ is an {\it exceptional graph} with root $r$ if there exists a sequence   $(G_0,x_0,y_0), \ldots, (G_n,x_n,y_n)$ for some nonnegative integer $n$ such that $G_0=K_{2,3}$, $G_n=G$, and for each $i$, the vertices $r$, $x_i$, and $y_i$ are distinct $2$-vertices of $G_i$, and 
\begin{itemize}
\item[(1)] if $x_iy_i\not\in E(G_i)$, then 
$G_{i+1}$ is obtained from $G_i$ by adding a path of length $3$ joining $x_i$ and $y_i$,
\item[(2)] if $x_iy_i\in E(G_i)$, then 
$G_{i+1}$ is obtained from $G_i$ by adding either a new vertex $v$ with $N_{G_{i+1}}(v)=N_{G_i}(w)$ for some $w\in\{x_i,y_i\}$ or a $4$-cycle $abcda$ and edges $ax_i, cy_i$. 
\end{itemize}
\end{definition}
Note that each $(G_i,x_i,y_i)$ is an exceptional graph with root $r$, and it has exactly three $2$-vertices $r$, $x_i$, $y_i$. 
By definition, 
if $x_i$ and $y_i$ are not adjacent, then they share the same neighbors.
Furthermore, if $i \ge 1$, then $r$ is adjacent to neither $x_i$ nor $y_i$ and has no common neighbors with $x_i$ or $y_i$. Hence, every exceptional graph other than $K_{2,3}$ has a unique root $r$ and the distance from $r$ to each of $x, y$ is at least $3$.

\subsection{Graphs with a cycle of length divisible by $4$}

The bound of $|V_2(G)|\le 2$ in Theorem~\ref{DLS93thm} is tight, as there exist a graph $G$ with $\delta(G)\ge 2$ and $|V_2(G)|=3$ that has  no $\modfour{0}$-cycles.
Nevertheless, we go one step further and characterize all such graphs, which we call {\it special}. 
% with $\delta(G)\ge 2$ and $|V_2(G)|=3$ that do not contain a $\modfour{0}$-cycle. 
As a consequence, analogous to the $\modthree{0}$-cycle case, we expand the class of graphs guaranteed to have a $\modfour{0}$-cycle, apart from a finite family of small graphs.
The following is our second main result.

\begin{theorem}\label{thm:main:0mod4}
    Let $G$ be a graph with $\delta(G) \ge 2$ and $|V_2(G)| \le 3$.
    Then $G$ has a $\modfour{0}$-cycle if and only if $G$ is not isomorphic to one of the graphs $T_1,T_2,\ldots,T_5$ in Figure~\ref{fig:4cycle}.
\end{theorem}

\begin{figure}[h!]
\centering
\includegraphics[page=3,height=3.3cm]{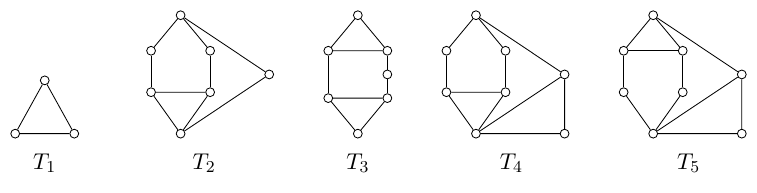}
\caption{Graphs without $\modfour{0}$-cycles}
\label{fig:4cycle}
\end{figure}

Since each graph $T_1,T_2,\ldots,T_5$ in Figure~\ref{fig:4cycle} has at most eight vertices, the following corollary is immediate.

\begin{corollary}
    Let $G$ be a graph with $\delta(G)\ge 2$.
    If $|V(G)|\ge 9$ and $|V_2(G)|\le 3$, then $G$ has a $\modfour{0}$-cycle.
\end{corollary}

\section{Preliminaries}

In this section, we introduce notation used throughout the paper and list several  facts that will be used in the proofs.

Given a walk $W$, its {\it{length}}, denoted  $\ell(W)$, is the number of edges on $W$.
For a path $P: v_0v_1 \ldots v_m$, let $P[v_i,v_j]$ denote the path $v_iv_{i+1}\ldots v_j$. 
For a cycle $C: v_0v_1\ldots v_mv_0$, let $C[v_{i},v_{j}]$ denote the path $v_i v_{i+1}\ldots v_{j}$, where the subscripts are taken modulo $m+1$. 
For two paths or cycles $P:v_1\ldots v_n$ and $Q:u_1\ldots u_m$ in a graph $G$ where $v_nu_1$ is an edge, let $P+Q$ denote the walk that follows  $P$ and then $Q$, that is, $P+Q:v_1\ldots v_n u_1\ldots u_m$.  

For two distinct vertices $x,y$ in a graph $G$, let $L_G(x,y)$ denote the set of lengths of all $(x,y)$-paths in $G$. 
For two sets $A$ and $B$ of integers, we write $A \subseteq B \pmod{k}$ if for each $a\in A$, there exists $b\in B$ such that $a\equiv b \pmod{k}$.
We write $A \equiv B \pmod{k}$ if $A \subseteq B \pmod{k}$ and $B \subseteq A \pmod{k}$.
 % \ic{no need to define $A\equiv B\pmod k$?k} \bp{We use the notation in Lemma 3.2.}
 % \ic{yes, i am saying we need to define $A\equiv B\pmod k$}
We also write $a\in B \pmod{k}$ if there is $b\in B$ such that $a \equiv b \pmod{k}$.

For two vertex sets $X$ and $Y$, a path $P$ is called an {\it $(X,Y)$-path} if one end of $P$ is in $X$, the other end is in $Y$, and no interior vertex of $P$ is in $X\cup Y$. 
When $X$ or $Y$ is a singleton, we drop the set notation for convenience. 
For example, if $X=\{x\}$ and $Y=\{y\}$, then we write an $(x,y)$-path. 
For nonempty vertex sets $X$ and $Y$, we say that a vertex set $S$ {\it separates} $X$ and $Y$ if every $(X,Y)$-path contains a vertex in $S$, and that $S$ is an {\it $(X,Y)$-separating set}. 

The following is one of Menger's well-known theorems. 

\begin{theorem}[\!\!{\cite[Theorem~2.2]{OBW13}}] \label{thm:menger} 
Let $G$ be a graph, $X$ and $Y$ be subsets of $V(G)$. For every positive integer $k$, there are $k$ pairwise vertex-disjoint $(X,Y)$-paths in $G$ if and only if every $(X,Y)$-separating set contains at least $k$ vertices.  
\end{theorem}

We will use the following simple facts in the proofs of our main results.

\begin{fact}\label{fact:2conn} \rm
Let $G$ be a graph with $\delta(G)\ge 2$  and $|V_2(G)|\le 3$.
If $G$ is not $2$-connected, then  $\delta(B)\ge 2$  and $|V_2(B)|\le 2$ for some end block $B$ of $G$. % IC: block is a graph
\end{fact}

A vertex cut $S$ of $G$ is {\it essential} if $G-S$ has at least two nontrivial components. 
A connected graph is {\it essentially $k$-connected} if it contains no essential vertex cut of size at most $k-1$.

\begin{fact}\label{fact:essen} \rm
Let $G$ be a $2$-connected graph such that $V_2(G)$ is an independent set, and $H$ be the graph obtained from $G$ by suppressing all $2$-vertices of $G$. If $H$ is a simple $3$-connected graph, then $G$ is essentially $3$-connected.
\end{fact}

\section{Proof of Theorem~\ref{thm:main:0mod3}}

Two distinct $2$-vertices with the same neighborhood are called {\it $2$-twins}.
We will use the following structure theorem of Gauthier \cite{Gauthier} for $2$-connected graphs with no   $\modthree{0}$-cycles.

\begin{theorem}[\!\!\cite{Gauthier}]\label{thm:Gauthier}
    Let $G$ be a $2$-connected graph with no $\modthree{0}$-cycles.
    Then $G$ has $2$-twins or two adjacent $2$-vertices.
\end{theorem}

For simplicity, we 
restate the operations in  Definition~\ref{def:exceptional} as follows. 
For a graph $H$ with $2$-vertices $u$ and $v$ (see Figure~\ref{fig:C-F-P}):
\begin{itemize} 
\item $P(H;u,v)$ is the graph obtained from $H$ by adding a path of length $3$ joining $u$ and $v$, 
    \item $C(H;u)$ is the graph obtained from $H$ by adding a vertex $w$ so that $u$ and $w$ are $2$-twins,
\item $F(H;u,v)$ is the graph obtained from $H$ by adding a $4$-cycle $abcda$ and two edges $ua,cv$. 
\end{itemize}
Thus, if $G$ is an exceptional graph with root $r$ and $|V(G)|>5$, then there is an exceptional graph $H$ with root $r$ such that $G$ is one of 
$P(H;u,v)$, $C(H;u)$, $F(H;u,v)$ where $\{u,v\}=V_2(H)\setminus\{r\}$. 
Note that $P(H;u,v)$ (resp.\ $C(H;u)$ and $F(H;u,v)$) corresponds to the operation (1) (resp.\ (2)) of Definition~\ref{def:exceptional}.
\begin{figure}[h!]
\centering
\includegraphics[page=1,height=3.5cm]{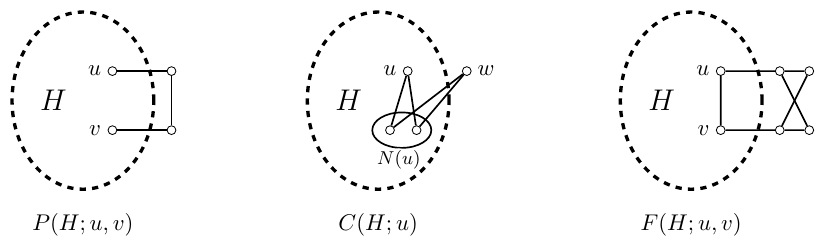}
\caption{Illustrations of $P(H;u,v)$, $C(H;u)$, $F(H;u,v)$.}\label{fig:C-F-P}
\end{figure}

We begin by establishing several structural properties of exceptional graphs.

\begin{lemma}\label{lem:path:properties}
Let $G$ be an exceptional graph with root $r$ and $\{x,y\}=V_2(G)\setminus\{r\}$. Then one of the following holds.
\begin{itemize}
\item[\rm(i)] $x$ and $y$ are $2$-twins, and $L_G(x,y) \equiv \{1,2\} \pmod{3}$.
\item[\rm(ii)] $x$ and $y$ are adjacent, and $L_G(x,y) \equiv \{0,1\} \pmod{3}$.
\end{itemize}
Moreover, if $G\neq K_{2,3}$, then $\{0,2\} \subseteq L_{G-y}(r,x) \pmod{3}$ and $\{0,2\} \subseteq L_{G-x}(r,y) \pmod{3}$.
\end{lemma}

\begin{proof} 
We use induction on $|V(G)|$. 
If $G=K_{2,3}$, then $x$ and $y$ are $2$-twins and the statement holds. 
Otherwise, $G\neq K_{2,3}$, and  there is an exceptional graph $H$ with root $r$ such that $G$ is one of 
$P(H;u,v)$, $C(H;u)$, $F(H;u,v)$ where $\{u,v\}=V_2(H)\setminus\{r\}$.

Suppose that $x$ and $y$ are adjacent.
Then $G=P(H;u,v)$ where $uxyv$ is a path of length $3$ and $uv$ is not an edge.
Moreover, $u$ and $v$ are $2$-twins in $H$. 
Then each $(x,y)$-path in $G$ is of the form $xy$ or $x+P'+y$ where $P'$ is a $(u,v)$-path in $H$.
Thus $L_G(x,y)=\{1\} \cup \{\ell+2 \colon\, \ell \in L_H(u,v)\}$.
By the induction hypothesis, $L_H(u,v)\equiv \{1,2\} \pmod{3}$.
Therefore $L_G(x,y)\equiv \{0,1\}\pmod{3}$.

Suppose that $x$ and $y$ are not adjacent. 
Then $x,y$ are $2$-twins and  $G$ is either $C(H;u)$  or $F(H;u,v)$.
Moreover, $u$ and $v$ are adjacent in $H$.
Consider the case $G=C(H;u)$. 
Without loss of generality, assume $u=x$ and let $N_H(u)=\{v, v'\}$.
In $G$, $xvy$ and $xv'y$ are $(x,y)$-paths of length $2$.
Each $(x,y)$-path in $G$ with length at least $3$ is of the form $u+P'+y$ where $P'$ is a path between $v$ and $v'$ in $H$ not using $u$.  
Thus $L_G(x,y)=\{2\}\cup\{\ell+1 \colon\, \ell \in L_H(u,v)\}$.
Now, consider the case $G=F(H;u,v)$.  
Each $(x,y)$-path in $G$ with length at least $3$ is of the form of $xzu+P'+vz'y$ where $N_G(x)=\{z,z'\}$ and $P'$ is a $(u,v)$-path in $H$, so $L_G(x,y)=\{2\} \cup \{\ell+4 \colon\, \ell \in L_H(u,v)\}$.
In both cases, by the induction hypothesis, $L_H(u,v)\equiv \{0,1\} \pmod{3}$.
Therefore $L_G(x,y)\equiv \{1,2\}\pmod{3}$.

We prove the moreover part by induction on $|V(G)|$. 
Suppose $G\neq K_{2,3}$. 
 Then there is an exceptional graph $H$ with root $r$ such that $G$ is one of  
 $C(H;u)$, $F(H;u,v)$, $P(H;u,v)$, where $\{u,v\}= V_2(H)\setminus\{r\}$.
 By the induction hypothesis, for each $a\in \{u,v\}$, there are $(r,a)$-paths $Q_a$ and $R_a$ in $H-b$ of length $\modthree{0}$ and $\modthree{2}$, respectively, where $\{a,b\}=\{u,v\}$.
 If $G=C(H;u)$, then we may assume $u=x$ and so $Q_u$ and $R_u$ are the desired paths.
If $G=F(H;u,v)$, where $w, w'$ are the common neighbors of $x$ and $y$ in $G$ such that $w \in N(u)\cap N(x)$ and $w' \in N(v)\cap N(y)$, then it follows that $Q_u+vw'x$ and $R_u+vw'x$ are the desired paths.  
 If $G=P(H;u,v)$, where $ux\in E(G)$, then $Q_v+ wux$ and $R_v+w'ux$ are the desired paths, where $w$ and $w'$ are common neighbors of $u$ and $v$ such that $w\not\in V(Q_v)$ and $w'\not\in V(R_v)$; note that $w, w'$ need not be distinct. 
 \end{proof}

\begin{lemma}\label{lem:exceptional} An exceptional graph has no $\modthree{0}$-cycles.
\end{lemma}

\begin{proof}
Let $G$ be an exceptional graph with root $r$, and let $x$ and $y$ be the other $2$-vertices of $G$.
We use induction on $|V(G)|$. 
It is easy to check $K_{2,3}$ has no $\modthree{0}$-cycles. 
Suppose $G\neq K_{2,3}$, so there is an exceptional graph $H$ with root $r$ such that 
$G$ is one of $C(H;u)$, $F(H;u,v)$, $P(H;u,v)$,  where $\{u,v\}=V_2(H)\setminus\{r\}$.
By the induction hypothesis, $H$ has no $\modthree{0}$-cycles. 
If $G=C(H;u)$, say $u=x$, then $G$ has no $\modthree{0}$-cycles, since besides the cycle of length $4$ containing $x,y$, every cycle of $G$ either lies in $H$ or can be transformed into a cycle of $H$ with the same length by replacing $y$ with $u$.
Otherwise, $G=F(H;u,v)$ or $G=P(H;u,v)$. 
Suppose to the contrary that $G$ has a $\modthree{0}$-cycle $C$. 
Then $C$ must go through both $u$ and $v$.
Moreover, $C$ has a vertex in $V(H)\setminus\{u,v\}$ and a vertex in $V(G)\setminus V(H)$.
Let $Q:=C[u,v]$ and $R:=C[v,u]$ be paths such that their internal vertices are in $H$ and $V(G)\setminus V(H)$, respectively, so $\ell(C)=\ell(Q)+\ell(R)$.
If $u$ and $v$ are $2$-twins in $H$, then $G=P(H;u,v)$, $\ell(R)=3$, and, by Lemma~\ref{lem:path:properties}(i), $\ell(Q)\equiv \{1,2\} \pmod{3}$. 
If $u$ and $v$ are adjacent in $H$, then $G=F(H;u,v)$, $\ell(R)=4$, and, by Lemma~\ref{lem:path:properties}(ii), $\ell(Q)\equiv \{0,1\} \pmod{3}$. 
Thus, in any case, $\ell(C)=\ell(Q)+\ell(R)\not\equiv 0 \pmod{3}$, a contradiction.
\end{proof}

Now we prove Theorem~\ref{thm:main:0mod3}.

\begin{proof}[Proof of Theorem~\ref{thm:main:0mod3}]
By Lemma~\ref{lem:exceptional}, it follows that if $G$ has a $\modthree{0}$-cycle, then $G$ is not an exceptional graph. 
For the converse, let $G$ be a non-exceptional graph with the minimum number of vertices such that $\delta(G)\ge 2$, $ |V_2(G)|\le 3$, and without $\modthree{0}$-cycles.
Then $G$ is connected and it is easy to check that $|V(G)|\ge 6$.
By the minimality of $G$ and Fact~\ref{fact:2conn}, we may assume $G$ is $2$-connected.

\begin{claim}\label{claim:no:twin}
    $G$ has no $2$-twins. 
\end{claim}
\begin{proof}
Suppose $G$ has $2$-twins $x$ and $y$.
    Let $N_G(x)=N_G(y)=\{x',y'\}$ so $x'$ and $y'$ are not adjacent.
    If $N_G(x)$ has a $2$-vertex $z$, then the vertex in $N_G(x)\setminus\{z\}$ is a cut vertex of $G$, a contradiction.
    Thus $d_G(x'),d_G(y')\ge 3$.
    If $G-y$ is not $2$-connected, then $x'xy'$ is the only $(x',y')$-path in $G-y$, so $x'$ or $y'$ is a cut vertex of $G$, which contradicts that $G$ is $2$-connected. 
    Therefore $G-y$ is $2$-connected.
    If $d_G(x')\ge 4$, then $\delta(G-y)\ge 2$ and $|V_2(G-y)|\le 3$.
    Since $G-y$ has no $\modthree{0}$-cycles, by the minimality of $G$, $G-y$ is an exceptional graph with root $r$. 
    Note that since $G-y$ is an exceptional graph, $d_G(y')=3$.
    Moreover, $r\neq x$, since $x$ has a  neighbor, namely $y'$, with degree $2$ in $G-y$. 
    Then $G=C(G-y;x)$, so $G$ is an exceptional graph, a contradiction.  
    By symmetry, $d_G(x')=d_G(y')= 3$.
    Let $x''$ and $y''$ be the neighbors of $x'$ and $y'$, respectively, other than $x$ and $y$. 
    Note that $x'' \neq y''$ since $G$ is $2$-connected. 

    Take a longest path $P$ in $G-y$ satisfying $x\in V(P) \subseteq V_2(G-y)$.
    Since $d_G(x')=3=d_G(y')$, we know $\{x',x,y'\}\subseteq V(P)$. 
    Let $G'=G-y-V(P)$. 
    Then since $G$ is $2$-connected, $\delta(G')\ge 2$.
    Moreover, $|V_2(G')| \le  6-|V(P)|$. 
    If $|V(P)|\ge 4$, then $|V_2(G')|\le 2$.
    Since $G'$ cannot be an exceptional graph, $G'$ has a $\modthree{0}$-cycle by the minimality of $G$, a contradiction. 
    Otherwise, $|V(P)|=3$, that is, $P:x'xy'$. 
    Since $\delta(G')\ge 2$, $|V_2(G')|\le 3$, and $G'$ has no $\modthree{0}$-cycles, it follows that $G'$ is an exceptional graph with root $r$ by the minimality of $G$. 
    Moreover, $d_G(x'')=d_G(y'')=3$, so $G=F(G';x'',y'')$. 
    
    Showing $x''$ and $y''$ are adjacent to each other would imply $G$ is an exceptional graph, a contradiction. 
    If $G'$ has an $(x'',y'')$-path $Q$ of length $\modthree{2}$, then  $Q+y'xx'x''$ is a $\modthree{0}$-cycle in $G$, a contradiction. 
    Thus, $G'$ has no $(x'',y'')$-paths of length $\modthree{2}$, and in particular, 
    $x''$ and $y''$ are not $2$-twins in $G'$ and $G'\ne K_{2,3}$. 
    By the moreover part of Lemma~\ref{lem:path:properties},
    $r\not\in \{ x'', y''\}$.
    Then $x''$ and $y''$ are $2$-vertices in an exceptional graph other than the root, and they are not $2$-twins. 
    Thus $x''$ and $y''$ are adjacent to each other, a contradiction.
\end{proof}

By Theorem~\ref{thm:Gauthier} and Claim~\ref{claim:no:twin}, $G$ has two adjacent $2$-vertices $x$ and $y$.
Let $x'$ and $y'$ be the neighbor of $x$ and $y$, respectively, other than $y$ and $x$.
Since $G$ is $2$-connected, $x'\neq y'$.
Suppose to the contrary that $y'$ is a $2$-vertex, so $V_2(G)=\{x,y,y'\}$.
 Since $G$ is $2$-connected, $x'$ cannot be a neighbor of $y'$.
Thus the subgraph $G':=G-\{x,y,y'\}$ of $G$ satisfies $\delta(G')\ge 2$ and $|V_2(G')|\le 2$, so $G'$ cannot be an exceptional graph.
Now, by the minimality of $G$, $G'$ has a $\modthree{0}$-cycle, a contradiction.

Thus $d_G(y')\ge 3$ and similarly, $d_G(x')\ge 3$.
Let $H=G-\{x,y\}$, so $\delta(H)\ge 2$ and $|V_2(H)|\le 3$. 
Since $H$ has no $\modthree{0}$-cycle, by the minimality of $G$, $H$ is an exceptional graph with root $r$, and $\{x',y'\} \subset V_2(H)$.
    We may assume that $H \neq K_{2,3}$ since otherwise $G=P(K_{2,3};x',y')$ is an exceptional graph.
    % and note that $G'\neq K_{2,3}$ and $\{x',y'\} \subset V_2(G')$. 
    If $x'$ and $y'$ are adjacent to each other or $r \in \{x',y'\}$, then by Lemma~\ref{lem:path:properties}, there is an $(x',y')$-path $P$ of length $\modthree{0}$ in $G'$, so $P+yxx'$ is a $\modthree{0}$-cycle in $G$, a contradiction. 
    Otherwise, $x'$ and $y'$ are not adjacent to each other and $r \not\in \{x', y'\}$, so $G=P(H;x',y')$ is an exceptional graph, a contradiction.
\end{proof}

\section{Proof of Theorem~\ref{thm:main:0mod4}}

We begin by collecting several known results on graphs without $\modfour{0}$-cycles. 
In \cite{0mod4cycle2025}, Gy\H{o}ri, Li, Salia, Tompkins, Varga, and Zhu showed that such graphs are planar.

\begin{lemma}[\!\!{\cite[Lemma 2]{0mod4cycle2025}}]\label{lem:planar}
Every non-planar graph  contains a $\modfour{0}$-cycle.
\end{lemma}

In the same paper, the structure between odd cycles was described.

\begin{lemma}[\!\!\cite{0mod4cycle2025}]\label{lem:previous}
Let $C_1$, $C_2$, and $C_3$ be odd cycles of a graph $G$ with $\ell(C_1) \equiv \ell(C_2) { \equiv \ell(C_3)}  \pmod 4$.
\begin{itemize}
\item[\rm(i)] {\rm{(Lemma 6(3))}}
 If $C_1$, $C_2$ are vertex-disjoint  and $P$, $Q$, $R$ are vertex-disjoint $(V(C_1),V(C_2))$-paths, then $C_1 \cup C_2 \cup P \cup Q \cup R$ contains a $\modfour{0}$-cycle.
\item[\rm(ii)] {\rm{(Lemma 7)}} Suppose that $C_1$, $C_2$, and $C_3$ pairwise intersect in exactly one common vertex. Let $Q_i$ be a $(V(C_i),V(C_{i+1}))$-path in $G-V(C_{i+2})$ for each $i\in \{1,2,3\}$ (the subscripts are taken modulo $3$) such that $Q_1$, $Q_2$, and $Q_3$ are pairwise internally vertex-disjoint. 
Then $C_1 \cup C_2 \cup C_3 \cup Q_1 \cup Q_2 \cup Q_3$ contains a $\modfour{0}$-cycle. 
\end{itemize}
\end{lemma}

In \cite{CPR2025}, Chu, Park, and Ryu showed that, under the assumption of $2$-connectivity, the structure between odd cycles satisfies additional properties.

\begin{lemma}[\!\!{\cite[Lemma 2.7(iv)]{CPR2025}}]\label{lem:odd:cycles} 
Let $G$ be a $2$-connected graph without $\modfour{0}$-cycles, and $C_1$, $C_2$, and $C_3$ be three edge-disjoint odd cycles of $G$.
Suppose that $|V(C_i)\cap V(C_j)|=1$ for every distinct $i,j\in\{1,2,3\}$. Then $V(C_1) \cap V(C_2) \cap V(C_3)=\{v\}$ for some vertex $v$.
Moreover, if $\ell(C_1)\equiv\ell(C_2)\equiv \ell(C_3)\pmod{4}$, then there is no connected subgraph $H$ of $G-v$ such that 
$V(H)\cap V(C_i)\neq \emptyset$ and $E(H)\cap E(C_i)=\emptyset$ for each $i\in \{1,2,3\}$. 
\end{lemma}

Analogous to the exceptional graphs in the $(0 \bmod 3)$-cycle case, we introduce a family of special graphs. 
% for $(0 \bmod 4)$. 
A graph is called {\it special} if it is isomorphic to $T_i$ for some $i \in \{1,2,3,4,5\}$, where $T_1, T_2, \ldots, T_5$ are the graphs in Figure~\ref{fig:4cycle}. The following observation lists some basic properties of special graphs, which can be easily verified.

\begin{observation}\label{obs:special}  Let $G$ be a special graph,  $u\in V_2(G)$ and $v\in V(G)\setminus\{u\}$. Then the following are true.  
\begin{enumerate}[\rm (i)]
    \item Every $3$-vertex is contained in a $\modfour{1}$-cycle.
        \item[\rm (ii)] $uv\in E(G)$ or 
        $0\in L_G(u,v)$. 
        \item[\rm (iii)] If $G\neq T_1$ and $v\in V_2(G)$, then
        $L_G(u,v)\equiv \{0,1,2,3\}\pmod{4}$.
    \item[\rm (iv)] 
    If $G\in \{T_4,T_5\}$ and $v \not\in V_2(G)$, then $2\in L_G(u,v)\pmod{4}$.
    \end{enumerate}
\end{observation}

Now, we prove Theorem~\ref{thm:main:0mod4}. Clearly, if $G$ has a $\modfour{0}$-cycle, then $G$ is not special. 
For the converse, let $G$ be a non-special graph with the minimum number of vertices such that $\delta(G)\ge 2$, $|V_2(G)|\le 3$, and without  $\modfour{0}$-cycles. 
Let $n=|V(G)|$.
Obviously, $G$ is connected and it is easy to check that $n\ge 6$.
By Theorem~\ref{DLS93thm}, we may assume that $|V_2(G)|=3$. 
By the minimality of $G$ and Fact~\ref{fact:2conn},  $G$ is $2$-connected.
We proceed by analyzing the structure of  $G$,  then use Euler's formula to derive a contradiction.

We first establish several local properties of $G$, which will be used to show that $G$ is essentially $3$-connected.
  
\begin{lemma}\label{Basic:0mod4} The following are true.
\begin{itemize}
    \item[\rm(i)] $V_2(G)$ is an independent set. Moreover, every neighbor of a $2$-vertex is a $3$-vertex.
    \item[\rm(ii)] No $2$-vertex is on a triangle.
    \item[\rm(iii)] There is no $3$-vertex adjacent to only $2$-vertices. 
\end{itemize}
\end{lemma}

\begin{proof}
(i) Let $v$ be a $2$-vertex. Take a longest path $P:x_0x_1\ldots x_k$ in $G$ such that $v$ is an internal vertex of $P$ and each internal vertex of $P$ is a $2$-vertex. 
Thus, $k\ge 2$ and $d_G(x_0),d_G(x_k)\ge 3$.  Since $G$ is $2$-connected, $x_0\neq x_k$. Let $G'=G-\{x_1,\ldots,x_{k-1}\}$. Then  $\delta(G')\ge 2$ and $|V_2(G')|\le 3-(k-1)+2=6-k$. 

We will show $k=2$. Suppose to the contrary that $k\ge 3$, so $|V_2(G')|\le 3$.
If $G'$ is not special, then $G'$ has a $\modfour{0}$-cycle by the minimality of $G$, a contradiction. 
Thus $G'$ is special and $|V_2(G')|=3$, so $k=3$. 
Moreover, $\{x_0,x_k\}\subset V_2(G')$, and 
$G$ is obtained from $G'$ by attaching $P$ to $x_0$ and $x_k$. 
This implies that $G$ has a $\modfour{0}$-cycle by Observation~\ref{obs:special}(iii), a contradiction. 
Therefore $k=2$, which implies that $V_2(G)$ is an independent set.

Thus, $P:x_0x_1x_2$ and $v=x_1$.
Suppose to the contrary that $d_G(x_0)\ge4$, so $|V_2(G')|\le 3$. 
If $G'$ is not special, then $G'$ has a $\modfour{0}$-cycle by the minimality of $G$, a contradiction.
Otherwise $G'$ is special, so $|V_2(G')|=3$, which further implies $x_2\in V_2(G')$.
By Observation~\ref{obs:special}(iv), $G'\not\in\{T_4,T_5\}$ and since $d_{G'}(x_0)\ge 3$, $G'\neq T_1$. 
Thus, $G'\in\{T_2, T_3\}$, implying $G$ has a $\modfour{0}$-cycle, or $G\in \{T_4, T_5\}$, a contradiction.
Hence $d_G(x_0)=d_G(x_2)=3$, implying that every neighbor of a $2$-vertex is a $3$-vertex.

(ii)     Suppose that $xyzx$ is a triangle where $x$ is a $2$-vertex.
Then $d_G(y)=d_G(z)=3$ by (i). Let $y'$ and $z'$ be the neighbor of $y$ and $z$, respectively, that is not $x$.
Since there are no $\modfour{0}$-cycles, $y'\neq z'$ and $y'z'$ is not an edge. 
% Note that since $G$ has no cycle of length $4$, $y'$ and $z'$ have at most one common neighbor.
If $d_G(y')=d_G(z')=2$, then consider $G'=G-\{x,y,z,y',z'\}$.
Since $G$ is $2$-connected and by the moreover part of (i), we know $\delta(G')\ge 2$ and $|V_2(G')|\le 2$.
By Theorem~\ref{DLS93thm}, $G'$ has a $\modfour{0}$-cycle, a contradiction.
If $d_G(y')=2$ and $d_G(z')>2$, then consider  $G''=G-\{x,y,z,y'\}$.
Let $y''$ be the neighbor of $y'$ that is not $y$.
By (i), $d_G(y'')=3$, so $\delta(G'')\ge 2$ and $|V_2(G'')|\le 3$.
By the minimality of $G$, $G''$ is special, so $\{y'',z'\} \subset V_2(G'')$. 
If $G''=T_1$, then $G=T_3$, a contradiction.
If $G'' \neq T_1$, then by Observation~\ref{obs:special}(iii), $G$ contains a $\modfour{0}$-cycle, a contradiction. 

Thus, $d_G(y'), d_G(z')\geq 3$. 
If $d_G(y')>3$, then consider $G'''=G-\{x,y,z\}$. 
By a similar argument, $G'''$ is special and $z'\in  V_2(G''')$.
Since $y'z'$ is not an edge, by Observation~\ref{obs:special}(ii), $G'''$ has a $(y',z')$-path of length $\modfour{0}$ so $G$ has a $\modfour{0}$-cycle, a contradiction.
By symmetry, we conclude $d_G(y')=d_G(z')=3$. 

Let $G^*$ be the graph obtained from $G-\{x,y,z\}$ by identifying $y'$ and $z'$ into a vertex $t$, and removing parallel edges and leaves. 
% Let $t$ be the identified vertex.   
Since $y'$ and $z'$ have at most one common neighbor, $\delta(G^*)\ge 2$ and $|V_2(G^*)|\le 3$.
Since $G$ has a $(y',z')$-path of length $4$, $G^*$ has no $\modfour{0}$-cycles.
Thus, by the minimality of $G$, $G^*$ is special, so $|V_2(G^*)|= 3$.
Moreover, $y'$ and $z'$ must have a common neighbor $w$ and $d_G(w)=3$. 
Thus $d_{G^*}(t)=3$, so $G^* \neq T_1$.
By Observation~\ref{obs:special}(i), there is a $\modfour{1}$-cycle $C$ passing through $t$; $C$ corresponds to a $(y', z')$-path of length $\modfour{1}$ in $G$. 
Combining this path with $y,z$ becomes a $\modfour{0}$-cycle in $G$, a contradiction.

(iii)  Suppose there is a $3$-vertex $v$ such that $N_G(v)=V_2(G)$.
Let $N_G(v)=\{x_1,x_2,x_3\}$.
For each $i$, let $y_i$ be the neighbor of $x_i$ that is not $v$, and by (i), $d_G(y_i)=3$.
Since $G$ has no $4$-cycles, $y_1, y_2,y_3$ are all distinct. 
Let $G'=G-\{v,x_1,x_2,x_3\}$, so $\delta(G')\ge 2$ and $V_2(G')=\{y_1,y_2,y_3\}$.  
By the minimality of $G$, $G'$ is special.
Observation~\ref{obs:special}(iii) implies $G'= T_1$, so $G= T_2$, a contradiction.
\end{proof}

We now show that $G$ is essentially $3$-connected. 

\begin{lemma}\label{lem:ess3conn}
 $G$ is essentially $3$-connected.
\end{lemma}
\begin{proof}  
Let $H$ be the graph obtained from $G$ by suppressing all $2$-vertices. 
Since $V_2(G)$ is independent by Lemma~\ref{Basic:0mod4}(i) and $G$ is $2$-connected, it suffices to show that $H$ is a simple $3$-connected graph by Fact~\ref{fact:essen}.
Clearly, $H$ has no loops.
Note that every edge in $H$ corresponds to either an edge of $G$ or a path of length $2$ in $G$. 
Since $G$ is a simple graph without $4$-cycles and no $2$-vertex lies on a triangle in $G$ by Lemma~\ref{Basic:0mod4}(ii), $H$ has no parallel edges. 
Thus $H$ is simple. 

Note that since $G$ is $2$-connected, $H$ is also $2$-connected. 
Suppose to the contrary that $H$ is not $3$-connected.
Then $H$ has a vertex cut $\{u,v\}$. Take a component $Q_1$ of $H-\{u,v\}$, and let 
\[
H_1 = H[V(Q_1)\cup\{u,v\}]-uv \quad \text{ and } \quad
H_2 = H-V(Q_1)-uv.
\]
For each $i\in\{1,2\}$, define $u_i$ as follows: if $u$ has degree $1$ in $H_i$, then let $u_i$ be the unique neighbor of $u$ in $H_i$ and replace $H_i$ with $H_i - u$; otherwise, let $u_i = u$. Define $v_i$ analogously. 
Then $\{u_i,v_i\}$ is a vertex cut in $H$ for each $i\in\{1,2\}$. 
Moreover, we may assume that $d_{H_i}(u_i)\ge 2$ and $d_{H_i}(v_i)\ge 2$ for each $i\in\{1,2\}$.  
Since $E(H_1)\cap E(H_2)=\emptyset$, without loss of generality, we may assume that $H_1$ contains at most one suppressed edge.  

Obtain $G_1$ from $H_1$ by restoring the suppressed edge (if any) with the corresponding path of length $2$.
Then $G_1$ is a subgraph of $G$ with 
$\delta(G_1)\ge 2$ and $|V_2(G_1)|\le 3$.
Since $G_1$ has no $\modfour{0}$-cycles, it follows from  Theorem~\ref{DLS93thm} that $|V_2(G_1)|=3$, so $\{u_1,v_1\} \subset V_2(G_1)$.
Moreover, by the minimality of $G$, $G_1$ is special.
    Note that since $|V(H_1)|>3$, $G_1 \neq T_1$.
    Since $H$ is $2$-connected, there is a $(u_1, v_1)$-path $P$ whose internal vertices are not in $H_1$, so the internal vertices of $P$ are not in $G_1$. 
    Thus, by Observation~\ref{obs:special}(iii), $G$ has a $\modfour{0}$-cycle, a contradiction.
    \end{proof}

By Lemma~\ref{lem:planar}, $G$ is planar, so fix an embedding of $G$ on the plane with no edge crossings; from now on, we assume $G$ is a plane graph. 
We will show that the number of small faces is  constrained, and then derive a contradiction by exploiting Euler's formula.

Let $f_i$ be the number of $i$-faces of $G$.
Since $G$ has no $\modfour{0}$-cycles, we have $f_4=0$.
Note that since $G$ is $2$-connected, the boundary of every face is a cycle. 
We begin by showing that $f_3 \le 1$ and $f_5 \le 4$, relying crucially on the fact that $G$ is essentially $3$-connected.

% \begin{lemma}\label{lem:3face}
%     $f_3\le 1$.
% \end{lemma}
% \begin{proof}
%     Suppose there are two $3$-faces in $G$.
%     By Lemmas~\ref{Basic:0mod4}(ii) and \ref{lem:ess3conn}, $G$ has no $2$-vertex  on a $3$-face and $G$ is essentially $3$-connected.
%     By Theorem~\ref{thm:menger}, there are three vertex-disjoint paths connecting the two $3$-faces; let their lengths be $a,b,c$. 
%     Now, $a+b\equiv a+c \equiv b+c \equiv 3 \pmod 4$, so $2(a+b+c)\equiv 1 \pmod 4$, a contradiction.
% \end{proof}

\begin{lemma}\label{lem:35face}
$f_3\le 1$ and $f_5\le 4$. Moreover, if $f_5\ge 3$, then $G$ has a vertex of degree at least $4$.
\end{lemma}
\begin{proof}
We first show $f_3 \le 1$.
Suppose there are two $3$-faces in $G$.
By Lemmas~\ref{Basic:0mod4}(ii) and \ref{lem:ess3conn}, $G$ has no $2$-vertex  on a $3$-face and $G$ is essentially $3$-connected.
By Theorem~\ref{thm:menger}, there are three vertex-disjoint paths connecting the two $3$-faces; let their lengths be $a,b,c$. 
Now, $a+b\equiv a+c \equiv b+c \equiv 3 \pmod 4$, so $2(a+b+c)\equiv 1 \pmod 4$, a contradiction.
Therefore $f_3 \le 1$.

We now bound $f_5$.
Suppose there are two vertex-disjoint $5$-cycles $C_1$ and $C_2$. 
Let $X$ be a smallest $(V(C_1),V(C_2))$-separating set.
%Then $|X| \ge 2$ since $G$ is $2$-connected. 
Since $G$ is essentially $3$-connected by Lemma~\ref{lem:ess3conn}, Theorem~\ref{thm:menger} implies there are three vertex-disjoint $(V(C_1),V(C_2))$-paths.
Lemma~\ref{lem:previous}(i) further implies $G$ has a $\modfour{0}$-cycle, which is a contradiction.
Therefore, $G$ has no vertex-disjoint $5$-cycles.
Since $G$ has no $\modfour{0}$-cycles, every two $5$-cycles intersect in exactly one vertex or a path of length $2$.  

\begin{claim}\label{clm:5faceatv}
There are no three $5$-faces such that every two of them intersect in exactly one vertex.
\end{claim}
\begin{proof}
Suppose there are three $5$-faces $C_1$, $C_2$, $C_3$ such that every two intersect in exactly one vertex.
By Lemma~\ref{lem:odd:cycles}, $V(C_1) \cap V(C_2) \cap V(C_3)=\{v\}$ for some vertex $v$. 
Let $B_i=C_i-v$ for each $i\in\{1,2,3\}$.
Suppose to the contrary that $V(C_3)$ is a $(V(B_1), V(B_2))$-separating set. 
Since $G$ is a plane graph, $C_3$ has a vertex $u$ different from $v$ such that 
for every component $H$ of  $G-(V(C_1)\cup V(C_2) \cup V(C_3))$, 
 $H$ does not have a neighbor in $V(C_3[u,v])\setminus\{u,v\}$ 
 or
 $H$ does not have a neighbor in 
 $V(C_3[v,u])\setminus\{u,v\}$; see Figure~\ref{fig:I}. 
\begin{figure}[h!]
\centering
\includegraphics[page=4,height=6cm]{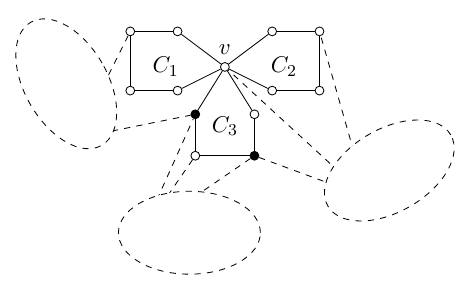}
\caption{An illustration of the case where $V(C_3)$ is a $(V(B_1), V(B_2))$-separating set (each dashed ellipse represents a component $H$ of $G - (V(C_1)\cup V(C_2)\cup V(C_3))$), and one of the black vertices can be the vertex $u$. }\label{fig:I}
\end{figure} 
Therefore, 
$\{u,v\}$ is a vertex cut of $G$ whose removal leaves at least two nontrivial components, a contradiction to Lemma~\ref{lem:ess3conn}.
Thus, by symmetry, there is a $(V(B_i), V(B_{i+1}))$-path $Q_i$ in $G-V(C_{i+2})$ for each $i\in\{1,2,3\}$.
By Lemma~\ref{lem:previous}(ii), $Q_1$ and $Q_2$ intersect in a vertex not in $V(C_1)\cup V(C_2)\cup V(C_3)$, so $H:=Q_1\cup Q_2$ is a connected subgraph of $G-v$ such that $V(H)\cap V(C_i)\neq \emptyset$ and 
$E(H)\cap E(C_i)=\emptyset$ for each $i\in \{1,2,3\}$.
By Lemma~\ref{lem:odd:cycles}, $G$ has a $\modfour{0}$-cycle, which is a contradiction.
\end{proof}

\begin{claim}\label{clm:5fat3p}
There are no three $5$-faces such that every two of them intersect in a path of length $2$. 
\end{claim}

\begin{proof}
Suppose there are three $5$-faces $C_1$, $C_2$,  $C_3$ such that every two of them intersect in a path of length $2$. 
Let $uvw$ be the path of length $2$ in which $C_1$ and $C_2$ intersect, and let $C_1:uvw w_1u_1u$ and $C_2:uvw w_2u_2u$. 
Since $C_1$ and $C_3$ intersect in a path of length $2$, without loss of generality, we may assume that $w \in V(C_3)$, which further implies  $C_3 : ww_1u_1u_2w_2w$.
Now, $w$ is a $3$-vertex  adjacent to only $2$-vertices, a contradiction to Lemma~\ref{Basic:0mod4}(iii). 
\end{proof}

% Recall that every two $5$-cycles intersect at a vertex or a path of length $2$. 
Let $A$ be an auxiliary graph where each vertex corresponds to a $5$-face of $G$ and each edge $pq$ is colored red and blue if the $5$-faces corresponding to $p$ and $q$ intersect in exactly one vertex and a path of length $2$, respectively. 
If $G$ has at least six $5$-faces, then since the Ramsey number $R(3,3)=6$, $A$ has a monochromatic $3$-cycle, which corresponds to three $5$-faces that pairwise intersect either in a vertex or in a path of length two.
This contradicts Claims~\ref{clm:5faceatv} and \ref{clm:5fat3p}, so  $f_5\le 5$. 

Suppose to the contrary that $f_5=5$.  
Since the extremal example for $R(3,3)$ is unique, there exist five $5$-faces $C_1$, $C_2, \ldots, C_5$ of $G$ such that $C_i$ and $C_{i+1}$ share a path of length two for each $i\in \{1,2,3,4,5\}$ (with indices taken modulo $5$).
% , and $C_i$ and $C_j$ intersect in exactly one vertex when $j\neq i\pm 1$. 
For each $i$, since $C_i$ and $C_{i+1}$ share a boundary path of length $2$, there exists a $2$-vertex on this path, so $|V_2(G)|\geq 5$.
This is a contradiction, hence, $f_5\le 4$. 
If $f_5\ge 3$, then by Claim~\ref{clm:5fat3p}, there are two $5$-faces sharing only vertex $v$, which is a vertex of degree at least $4$. 
\end{proof}

We now complete the proof of Theorem~\ref{thm:main:0mod4}.
Let $e(G)$ and $f(G)$ be the number of edges and faces, respectively, of $G$.
Since $f_3 \le 1$, $f_4=0$, and $f_5 \le 4$,  by Euler's formula,
\[ n+f(G) = 2+e(G)=2+\frac{1}{2}\sum_{i\ge 3}if_i \ge 2+ 3f(G) -\frac{3}{2}f_3-f_4-\frac{1}{2}f_5 \ge 3f(G)-\frac{3}{2},  \]
 so  $f(G) \le \frac{2n+3}{4}$. 
As $f(G)$ is an integer, $f(G) \le \frac{2n+2}{4}$.
Therefore $e(G) = n+ f(G)-2  \le \frac{3n-3}{2}$.
On the other hand, since $\delta(G)\ge 2$ and $|V_2(G)|\le 3$, $e(G) \ge \frac{3n-3}{2}$.
Thus $e(G)=\frac{3n-3}{2}$, which implies that all vertices of $G$ not in $V_2(G)$ are $3$-vertices.
In particular $f_5\leq 2$ by Lemma~\ref{lem:35face}.
Moreover,  
  \[ \frac{3n-3}{2} = e(G)=\frac{1}{2}\sum_{i\ge 3}if_i  \ge  3f(G) -\frac{3}{2}f_3-f_4-\frac{1}{2}f_5  =3(2+e(G)-n)-\frac{3}{2}f_3-\frac{1}{2}f_5  
  =  \frac{3n+3-3f_3-f_5}{2}.\] 
This yields $6\le 3f_3+f_5$, which is impossible. 
% Therefore $f_3=1$ and $f_5\ge 3$, a contradiction to Lemma~\ref{lem:35face}. 
This completes the proof of Theorem~\ref{thm:main:0mod4}.

\section*{Acknowledgement}
%%%%%%%%%%%%%%
Ilkyoo Choi was supported by the Hankuk University of Foreign Studies Research Fund, the National Research Foundation of Korea(NRF) grant funded by the Korea government(MSIT) (RS-2025-23324220), Institute for Basic Science (IBS-R029-C1), and the Korea Institute for Advanced Study (KIAS) grant funded by the Korea government.
Hojin Chu was supported by a KIAS Individual Grant (CG101801) at Korea Institute for Advanced Study. 
Ringi Kim was supported by the National Research Foundation of Korea(NRF) grant funded by the Korea government(MSIT) (No.~RS-2025-00561867), and supported by INHA UNIVERSITY Research Grant.
Boram Park was supported by the National Research Foundation of Korea(NRF) grant funded by the Korea government(MSIT) (No.~RS-2025-00523206), and supported by the New Faculty Startup Fund from Seoul National University.

\end{document}